\documentclass[12pt]{amsart}

\usepackage{graphicx}
\usepackage{amsmath,amsfonts,amssymb}
\usepackage{epsfig}

\usepackage{fullpage}

 \normalsize
 \allowdisplaybreaks[2]

\newcommand{\todo}[1]{\par \noindent
  \framebox{\begin{minipage}[c]{0.95 \textwidth} TO DO:
      #1 \end{minipage}}\par}
  
 \newcommand{\dni}{\partial_n^{-1}}
\newcommand{\R}{\mathbb{R}}
\newcommand{\Z}{\mathbb{Z}}
\newcommand{\C}{\mathcal{C}}
\newcommand{\K}{\mathcal{K}}

\newtheorem{theorem}{Theorem}[section]
\newtheorem{proposition}[theorem]{Proposition}
\newtheorem{corollary}[theorem]{Corollary}
\newtheorem{lemma}[theorem]{Lemma}
\newtheorem{definition}[theorem]{Definition}
\newtheorem{conjecture}[theorem]{Conjecture}
\theoremstyle{remark}

\newtheorem{example}[theorem]{Example}

\title{Compositions constrained by graph Laplacian minors}

\author[B. Braun]{Benjamin Braun}
\address{Department of Mathematics\\
         University of Kentucky\\
         Lexington, KY 40506--0027}
\email{benjamin.braun@uky.edu}

\author[R. Davis]{Robert Davis}
\email{davis.robert@uky.edu}

\author[A. Harrison]{Ashley Harrison}
\email{ashley.harrison0525@uky.edu}

\author[J. McKim]{Jessica McKim}
\email{jessica.mckim@uky.edu}

\author[J. Noll]{Jenna Noll}
\email{jenna.noll@uky.edu}

\author[C. Taylor]{Clifford Taylor}
\email{clifford.taylor@uky.edu}

\date{9 February, 2012}

\thanks{This paper is the result of a Summer 2011 graduate/undergraduate research experience.
All the authors were partially supported by the NSF through award DMS-0758321 of the first author.
The first author thanks Matthias Beck, Matthias Koeppe, Peter Paule, Carla Savage, and Zafeirakis Zafeirakopoulos for many interesting discussions during the American Institute of Mathematics SQuaREs on Polyhedral Geometry and Partition Theory.}



\begin{document}

\begin{abstract}
Motivated by examples of symmetrically constrained compositions, super convex partitions, and super convex compositions, we initiate the study of partitions and compositions constrained by graph Laplacian minors.
We provide a complete description of the multivariate generating functions for such compositions in the case of trees.
We answer a question due to Corteel, Savage, and Wilf regarding super convex compositions, which we describe as compositions constrained by Laplacian minors for cycles; we extend this solution to the study of compositions constrained by Laplacian minors of leafed cycles.
Connections are established and conjectured between compositions constrained by Laplacian minors of leafed cycles of prime length and algebraic/combinatorial properties of reflexive simplices.
\end{abstract}

\maketitle


\section{Introduction}

A \emph{partition} of a positive integer $n$ is a sequence $\lambda:=(\lambda_1,\lambda_2,\ldots,\lambda_k)$ of non-negative integers satisfying $\lambda_k\leq\lambda_{k-1}\leq \cdots \leq \lambda_1$ such that $\sum_i\lambda_i=n$.
We call each $\lambda_i$ a \emph{part} of $\lambda$ and say that $\lambda$ has $k$ \emph{parts}. 
A \emph{(weak) composition} of $n$ is a sequence $\lambda:=(\lambda_1,\lambda_2,\ldots,\lambda_k)$ of non-negative integers satisfying $\sum_i\lambda_i=n$.  
Beginning with \cite{andrewspa1} and through eleven subsequent papers, George Andrews and various coauthors successfully revived Percy MacMahon's technique of \emph{Partition Analysis} (PA) from obscurity as a tool for computing generating functions for partitions and compositions.
PA is particularly well-suited to the study of partitions and compositions constrained by a linear system of inequalities, i.e., of the form $A\lambda\geq b$ for some integral matrix $A$ and integral vector $b$; this is equivalent to the study of integer points in a rational polyhedral cone.
Motivated in part by this renewed interest in PA, there has been growing interest in the application of polyhedral-geometric techniques to the study of integer partitions and compositions.

For example, in \cite{PakGeomPart}, Igor Pak considers several partition theory problems from the point of view of lattice-point enumeration, leading to the establishment of unique bijective correspondences.
In \cite{beckgesselleesavage}, Matthias Beck, Ira Gessel, Sunyoung Lee, and Carla Savage consider the general case of determining integer-point transforms of symmetrically constrained families of compositions; their techniques include both unimodular decompositions of cones and the use of permutation statistics.
In \cite{BrightSavage}, Katie Bright and Carla Savage produce a bijection between certain ``boxes of width $t$'' of lecture hall partitions and integer points in the cube $[0,t]^d$; the methods for defining this bijection are motivated via geometry and again involve permutation statistics.
In \cite{BeckBraunLe}, the first author, Matthias Beck, and Nguyen Le give geometric proofs of several of the results in the PA papers by George Andrews et al; these proofs often produce multivariate generating functions in cases where PA produced only univariate identities.
Finally, motivated by \cite{beckgesselleesavage} and \cite{BrightSavage}, the first author and Matthias Beck use geometric techniques to establish new multivariate Euler-Mahonian distributions in \cite{BeckBraunEulerMahonian}.

This paper continues this line of investigation by initiating the explicit study of partitions and compositions constrained by graph Laplacian minors.
Algebraic graph theory is a rich source of interesting constraint matrices, the first being the signed incidence matrix of a graph, defined as follows.
Let $G=\left(\{0,1,\ldots,n\},E\right)$ be a simple graph with $n+1$ vertices and $|E|=m$ edges.
\begin{definition}
For any fixed edge orientation $\varepsilon$, the \emph{signed vertex-edge incidence matrix} of $G$ with respect to $\varepsilon$ is the $n+1\times m$ matrix $\partial$, with rows indexed by vertices of $G$ and columns indexed by edges of $G$, whose $(v,e)$-th entry is $1$ if $v$ is the positive end of $e$, $-1$ if $v$ is the negative end of $e$, and $0$ otherwise.
\end{definition}

The study of compositions with constraint matrices given by $\partial^T$ for various graphs was recently advanced in \cite{SavageEtAlDigraphEnumeration} by using five guidelines from MacMahon's partition analysis along with the Omega software package.
Additionally, such compositions often fall within the realm of the theory of $P$-partitions \cite[Chapter 4]{StanleyVol1}.
Our goal is to extend this line of investigation by using as constraint matrices various minors of graph Laplacians, defined as follows.
\begin{definition}
The \emph{vertex Laplacian} of $G$ is $L:=D-A$, where $D$ is the $n+1\times n+1$ matrix indexed by the vertices of $G$ such that the matrix entry $d_{i,j}$ is the degree of vertex $i$ if $i=j$ and $0$ otherwise and $A$ is the $0/1$-valued adjacency matrix of $G$ indexed in the same way as $D$.
Given a graph $G$, the \emph{$i$-th Laplacian minor} of $G$ is $L_i$, the Laplacian matrix $L$ with row $i$ and column $i$ deleted.
\end{definition}
The connection between the incidence matrix of a graph and its Laplacian is the well-known equation $L=\partial\partial^T$ \cite[Chapter 6]{BiggsAGT}.
Kirchoff's Matrix-Tree Theorem \cite{BiggsAGT} asserts that given a graph $G$ on $\{0,\ldots,n\}$ and given any $i$ between $0$ and $n$, $\det(L_i)$ is equal to the number of labeled spanning trees of $G$. 
This establishes a strong connection between Laplacian minors and enumerative properties of graphs.

There are several examples of partitions and compositions occurring in the literature on partition and composition enumeration that are constrained by graph Laplacian minors. 
The following three such examples provide further motivation for the study of these objects; note that these examples were not originally introduced in the context of graph Laplacian minors.

\begin{example}
Let $K_{n+1}$ be the complete graph on $\{0,1,\ldots,n\}$.
The compositions determined by $L_0\lambda\geq 0$ are one example of the class of \textit{symmetrically constrained compositions}.
A study of multivariate generating functions for symmetrically constrained compositions was given by Beck, Gessel, Lee, and Savage in \cite{beckgesselleesavage}.
\end{example}
\begin{example}
Let $P_n$ be the path graph with vertices $\{0,1,\ldots,n-1\}$.
The partitions determined by $L_0\lambda\geq 0$ are named \textit{super convex partitions} by Corteel, Savage, and Wilf in Example 3 of \cite{corteelsavagewilf}; they were first studied using techniques from MacMahon's Partition Analysis by Andrews in \cite{andrewspa1}, where they are named \textit{partitions with mixed difference conditions}.
The multivariate generating function for such partitions is described in \cite{corteelsavagewilf}.
\end{example}
\begin{example}\label{cycleconj}
Let $C_n$ be the cycle on vertices $\{0,1,\ldots, n-1\}$.
The compositions determined by $L_0\lambda\geq 0$ were introduced by Corteel, Savage, and Wilf in the concluding remarks of \cite{corteelsavagewilf}, where they are named \textit{super convex compositions}.
The determination of the multivariate generating function for super convex compositions was given as an open problem at the end of \cite{corteelsavagewilf}.
\end{example}

We now introduce our primary object of study.
\begin{definition}
Given a graph $G$ with Laplacian $L$ and a vertex $i$ of $G$, the \emph{compositions constrained by the $i$-th Laplacian minor of $G$} are those satisfying $L_i\lambda\geq 0$.
\end{definition}

Our paper is structured as follows.
In Section~\ref{cones}, we review techniques from integer point enumeration in polyhedral cones that will be useful for producing generating function identities.
In Section~\ref{trees}, we investigate compositions constrained by Laplacian minors for trees.
Theorem~\ref{laplacianminorthm} gives a combinatorial interpretation for the entries of the matrix inverse for tree Laplacian minors, leading to Corollary~\ref{treegenfnthm}, the main result of the section.

In Section~\ref{ncycles}, we provide a solution to the open problem of Corteel, Savage, and Wilf discussed in Example~\ref{cycleconj}.
However, as shown in the section, our solution is less clear than one might want.
Thus, motivated by the underlying graphical structure of our compositions, in Section~\ref{leafedncycles} we examine compositions constrained by Laplacian minors for leafed cycles, i.e. for cycles with an additional vertex forming a leaf appended.
These compositions are closely related to the super convex compositions in Example~\ref{cycleconj}.
The addition of a leaf to the cycle causes a change in the structure of the resulting compositions, leading to Theorem~\ref{leafedcyclethm}.
We also consider in Conjecture~\ref{mainconjecture} a possible combinatorial interpretation for the compositions constrained by leafed cycles, and we prove special cases of this conjecture.

In Section~\ref{reflexive}, we consider generating functions for certain subfamilies of compositions constrained by Laplacian minors for leafed cycles of prime length.
Surprisingly, these generating functions are Ehrhart series for a family of reflexive simplices, as shown in Theorem~\ref{reflexivethm}.
After introducing Ehrhart series and reflexive polytopes and proving Theorem~\ref{reflexivethm}, we indicate some possible connections between our enumerative problem and the algebraic/combinatorial structure of these simplices.

We close in Section~\ref{2powercycles} with a conjecture regarding compositions constrained by Laplacian minors for leafed cycles of length $2^k$.


\section{Cones and generating functions}\label{cones}

The problem of enumerating integer compositions and partitions satisfying a system of linear constraints $A\lambda \geq b$ is identical to that of enumerating integer points in the polytope or polyhedral cone defined by these constraints.
For further background related to integer point enumeration in polyhedra, see \cite{BeckRobinsCCD}.

Given a square matrix $A$ with $\det(A)\neq 0$, the solution set
\[
\C=\left\{x\in \R^n:Ax\geq 0\right\}
\]
is called an \emph{$n$-dimensional simplicial cone}.
The reason for the terminology ``simplicial'' comes from $\C$ being an $n$-dimensional cone with $n$ bounding hyperplanes, also called \emph{facets}.
The following lemma, whose proof is straightforward, indicates that simplicial cones can be described either by specifying these $n$ facets or by specifying the $n$ ray generators for the cone.

\begin{lemma}\label{facetraylemma}
Let A be an invertible matrix in $\mathbb{R}^{n\times n}$. Then
\[
\{x \in \mathbb{R}^n : Ax\geq 0\} = \{A^{-1}y : y\geq 0\}
\]
where each inequality is understood componentwise.
\end{lemma}

We encode the integer points in a cone $\C$ using the \emph{integer point transform} of $\C$ given by the multivariate generating function
\[
\sigma_\C(z_1,z_2,\ldots,z_n):=\sum_{m\in \C\cap \Z^n}z_1^{m_1}z_2^{m_2}\cdots z_n^{m_n} \, .
\]
This function encodes the integer points in $\C$ as a formal sum of monomials.
Throughout this paper, we will typically need to specialize $\sigma_\C$ in one of two ways.
If we set $z_1=z_2=\cdots =z_n=q$, we obtain
\[
\sigma_\C(q,q,\ldots,q)=\sigma_\C(q)=\sum_{m\in \C\cap \Z^n}q^{|m|} \, ,
\]
where $|m|=m_1+m_2+\cdots+m_n$.
Setting $z_1=q$ and $z_2=z_3=\cdots=z_n=1$, we obtain
\[
\sigma_\C(q,1,\ldots,1)=\sigma_\C(q)=\sum_{m\in \C\cap \Z^n}q^{m_1} \, .
\]

When $A$ is an integer matrix having $\det(A)=1$, the ray generating matrix $A^{-1}$ for the $n$-dimensional simplicial cone $\C=\left\{x\in \R^n:Ax\geq 0\right\}$ has integer entries; hence, every integer point in $\C$ can be expressed as $A^{-1}y$ for a unique integer vector $y\geq 0$.
In this case, the integer point transform can be simplified as a product of geometric series as follows.

\begin{lemma}
Let $A=(a_{i,j})$ be an integer matrix with $\det(A)=1$ and inverse matrix $A^{-1}=(b_{i,j})$.
Then
\[
\sigma_\C(z_1,z_2,\ldots,z_n):=\frac{1}{\prod_{j=1}^n(1-z_1^{b_{1,j}}z_2^{b_{2,j}}\cdots z_n^{b_{n,j}})  }
\]
\end{lemma}

More generally, suppose $\det(A) \neq 1$.
In this case, $\sigma_\C(z_1,z_2,\ldots,z_n)$ is still rational, but $A^{-1}$ need not have integer entries.
As a result, representations of integer points in $\C$ are not necessarily obtained by using integral combinations of the minimal integral vectors on the generating rays of $\C$.
We can work around this difficulty in the following manner.

Let $v_1, v_2, \ldots, v_n$ be the (integral) ray generators of $\C$.
Define the \emph{(half-open) fundamental parallelepiped} (FPP) to be 
\[\Pi = \{\lambda_1v_1 + \lambda_2v_2 + \cdots\lambda_nv_n \ : 0 \leq \lambda_1, \lambda_2, \ldots, \lambda_n < 1\} \, .
\]
The cone $\C$ can be tiled by integer translates of $\Pi$, but there may be integer points contained in $\Pi$ that can only be obtained using strictly rational scaling factors of the ray generators of $\C$.
This leads to our main tool for computing generating functions.
\begin{lemma}\label{fpplemma}
Using the preceding notation, if the points contained in $\Pi$ are $w_1,w_2,\ldots,w_m$, the integer point transform of $\C$ is given by
\begin{equation}
\sigma_\C(z_1,z_2,\ldots,z_n):=\frac{\sum_{i=1}^m z_1^{w_{i,1}}z_2^{w_{i,2}}\cdots z_n^{w_{i,n}}}{\prod_{j=1}^n(1-z_1^{v_{1,j}}z_2^{v_{2,j}}\cdots z_n^{v_{n,j}})  }.
\end{equation}
\end{lemma}


 \section{Trees}\label{trees}

We begin our study of compositions constrained by graph Laplacian minors with the special case of trees. 
Since a tree $T$ has a unique spanning tree, namely itself, the Matrix-Tree Theorem implies that $\det(L_i)=1$ for any vertex $i$. 
Thus, compositions constrained by tree Laplacian minors will have a generating function of the form
\[
\sigma_\C(z_1,z_2,\ldots,z_n)=\frac{1}{\prod_{j=1}^n(1-z_1^{b_{1,j}}z_2^{b_{2,j}}\cdots z_n^{b_{n,j}})  } \, .
\]
Since the $b_{k,j}$'s arise as the entries in the inverse matrix for $L_i$, we are interested in calculating this inverse.
We first observe that without loss of generality we may consider only those vertices $i$ which are leaves of $T$, through the following reduction.

Suppose that a vertex $i \in T$  is not a leaf. Then taking the minor at $i$ deletes row and column $i$ from the Laplacian matrix, which equates to removing vertex $i$ in the
graph. 
We note that removing the vertex $i$ will split our tree into some number of disjoint subtrees, which we label $s_1, \ldots,s_k$.
Label the vertices of $s_1$ consecutively from $r_1 + 1 , r_1 + 2, \ldots, r_2$, then label the vertices of $s_2$ consecutively as $r_2+1, r_2 + 2 \ldots,r_3$, and so on until labeling the vertices of $s_k$ as $r_{k}+1, \ldots,r_{k+1}$.
Now consider the $r_{k+1} \times r_{k+1}$ matrix $M$, indexed naturally.
$M$ has a block form resulting from the fact that the subtrees $s_1, \ldots,s_k$ are distinct and disconnected.
In particular, each subtree $s_j$, $1 \leq j \leq k$, will have a corresponding block matrix $M_j$ consisting of rows $r_{j}+1, \ldots, r_{j+1}$ and columns $r_{j}+1, \ldots, r_{j+1}$.

\tiny 

 \[
\Large \textbf{$M$} \ \  \ \textbf{=} \ \ \ \tiny \bordermatrix{
   & & & & & & & & & & & & &  \cr
   & & & &| & & & & & & & & &  \cr
   & & \mbox{\large $M_1$} & & |& & & & & & & & & \cr
   & & & &| & & & & & & && \mbox{\Large 0}&  \cr
   & \textbf{--} & \textbf{--}&\textbf{--} & \textbf{--}& \textbf{--}&\textbf{--} &\textbf{--} & & & & & &  \cr
   & & & &| & & & &| & & & & &  \cr
   & & & & |& & \mbox{\large $M_2$} & &| & & & & &  \cr
   & & & & |& & & & |& & & & &  \cr
   & & & & &\textbf{--} &\textbf{--} &\textbf{--} & & & & & &  \cr
   & & & & & & & & & \ddots & & & &  \cr
   & & & & & & & & & & & \textbf{--}& \textbf{--}&\textbf{--}  \cr
   & & &\mbox{\Large 0} & & & & & & & |& & &  \cr
   & & & & & & & & & & |& & \mbox{ \large $M_k$} & \cr
   & & & & & & & & & &| & & & }.
\]
\normalsize

Further, each $M_j$ will correspond to the Laplacian minor for $s_j$ where the minor is taken at $i$, which is now simply a leaf of $s_j$. 
By properties of block matrices, we have that $M^{-1}$ will be a block matrix with blocks $M_j^{-1}$.
Thus, to find $M^{-1}$, we compute the Laplacian minor inverse for each subtree.
With this observation, we can begin our search for a general form of the generating function of compositions constrained by Laplacian minors of trees.


\subsection{Inverses of Laplacian minors for trees}\label{laplacianinversetrees}

To determine the denominator of $\sigma_{\C}$ for an arbitrary tree $T$ minored at a leaf, we first find a combinatorial interpretation for the columns of the inverse of the Laplacian minor for $T$.
(While Theorem~\ref{laplacianminorthm} below seems likely to be known to experts in algebraic graph theory, we could not find an explicit statement of it in the literature.)
Throughout this section, we assume that $T$ is a tree on the vertex set $\{0,\ldots,n\}$ with an arbitrary orientation of the edges of $T$.

\begin{theorem}\label{laplacianminorthm}
Consider a tree $T$ and its corresponding Laplacian matrix $L$. Minor at a leaf $n$ of $T$ and let $L_n^{-1}=(l_{i,j})$. 
Then entry $l_{i,j}$ is the distance from $n$ to the path connecting the vertices $i$ and $j$. 
Equivalently, $l_{i,j}$ is the length of the intersection of the path from $n$ to $i$ with the path from $n$ to $j$.
\end{theorem}

Theorem~\ref{laplacianminorthm} has the following implication on the level of generating functions.
Let $\sigma_{\C}(q)$ denote the specialiazed generating function $\sigma_{\C}(q,q,\ldots,q)$.

\begin{corollary}\label{treegenfnthm}
The specialized generating function for a tree $T$ minored at leaf $n$ is given by
\[
\sigma_\C(q) = \frac{1}{\prod_{i=0}^{n-1} (1- q^{b_i})},
\]
where $b_i$ is the sum of the distances between $n$ and the paths connecting vertex $i$ with the other vertices of $T$.
\end{corollary}

To prove Theorem \ref{laplacianminorthm}, we return to the fact that $L=\partial\partial^T$ and proceed combinatorially.
We say the \emph{$n$-th subminor of $\partial$} (resp. of $\partial^{T}$) imposed by a vertex $n$ of a tree is the square matrix in which the $n$-th row, (resp. column) is deleted.
We denote the $n$-th subminor of $\partial$ as $\partial_{n}$, with the understanding that $\partial_n^T$ is $(\partial_n)^T$. 
It is easy to see that since $L=\partial\partial^{T}$, we have $L_{n}=\partial_{n}\partial_{n}^{T}$; thus, $L_{n}^{-1}=(\partial_{n}^{T})^{-1}\partial_{n}^{-1}$.

\begin{proposition}
Let $G_n=(g_{e,j})$, $0\leq j\leq n-1$, where $g_{e,j}$ is determined as follows: 
\begin{itemize}
\item $1$ if edge $e$ is on the path between $j$ and $n$ and $e$ points away from $n$;
\item $-1$ if  edge $e$ is on the path between $j$ and $n$ and $e$ points toward $n$;
\item $0$ if edge $e$ is not on the path between $j$ and $n$.
\end{itemize}
Then $G_{n}\partial_{n} = I$, hence $\dni=G_{n}$.
\end{proposition}

\begin{proof}
Set $G_{n}\partial_{n}=M=(m_{i,j})$ and first consider $m_{i,i}$. 
Entries of row $i$ of $\partial_n$ are non-zero for each edge, i.e., column, adjacent to $i$. 
The path connecting vertices $i$ and $n$ overlaps with exactly one of the edges adjacent to $i$, which is reflected in $G_n$. 
By construction, the entries in $\partial_n$ and $G_n$ corresponding to this vertex/edge pairing have the same sign. 
So $m_{i,i}=1$ for all $i$.

Now consider $m_{i,j}$ with $i\neq j$. 
The path from $j$ to $n$ overlaps with either no edges or exactly two edges connected to $i$. 
If there is no overlap, then the product of row $i$ in $G_n$ and column $j$ in $\partial_n$ is zero.  
If there is overlap, suppose both edges point either toward or away from $n$. 
Then the entries in $G_n$ have the same sign, but those in $\partial_n$ have different signs, thus summing to zero. 
A similar canceling occurs if the edges point in different directions. 
So, $m_{i,j}=0$ if $i\neq j$, implying $M=I$. 
Therefore,  $G_n=\dni$. 
\end{proof}

\begin{proof}[Proof of Theorem \ref{laplacianminorthm}]
Since $L_{n}^{-1}=(\partial_{n}^{T})^{-1}\partial_{n}^{-1}$, we may multiply columns $i$ and $j$ in $\dni$ component-wise to find $l_{i,j}$. 
Consider an arbitrary edge of the graph. 
If the edge belongs to only one of the paths from $i$ to $n$ or from $j$ to $n$, then there will be a zero corresponding to that edge in either column $i$ or $j$, contributing nothing to entry $l_{i,j}$. 
Otherwise we obtain a $1$, since the sign of a row for any edge is constant. 
Thus, summing the component-wise product will record the length of the path from $n$ to the path between $i$ and $j$. 
\end{proof}

Theorem \ref{laplacianminorthm} provides some insight into the form of the generating function for compositions constrained by an arbitrary tree. 
However, determining the lengths of paths between vertices in trees is not an easy exercise in general. 
In the next subsection, we analyze a special case.


\subsection{$k$-ary trees}\label{karytrees}

We next provide an explicit formula for a specialized generating function for compositions constrained by Laplacian minors of $k$-ary trees.
Recall that a $k$-ary tree is a tree formed from a root node by adding $k$ leaves to the root, then adding $k$ leaves to each of these vertices, etc., stopping after a finite number of iterations of this process.
To study the compositions constrained by a Laplacian minor of a $k$-ary tree $T$, we add a leaf labeled $0$ to the root of $T$ and minor at this new vertex.
We define the \emph{level} of a vertex as its distance to vertex $0$ and a \emph{subtree} of a vertex as the tree which includes the vertex and all of its children (we assume all edges are oriented away from vertex $0$). 
Recall that $[n]_q:=1+q+q^2+q^3+ \cdots +q^{n-1}$.

\begin{theorem}
For a $k$-ary tree with $n$ levels,
\begin{eqnarray*}
\sigma_\C(q)&=&\frac{1}{\displaystyle  \prod_{j=1}^{n} \left(1-q^{\textstyle j[n-j+1]_k+\sum_{i=1}^{j-1}(j-i)k^{n-(j-i)})}\right)^{\textstyle k^{j-1}}}\\
		    &=& \frac{1}{\displaystyle \prod_{j=1}^{n}\left(1-q^{ \textstyle j[n-j+1]_k+k^{n-j}(j([j]_k-1)-k\frac{d}{dk}([j]_k-1))}\right)^{\textstyle k^{j-1}}}
 \end{eqnarray*}

\end{theorem}

\begin{proof}
Suppose a $k$-ary tree $T$ is given.
In order to find the generating function consider the column sums of $L_0^{-1}$.
Choose a particular column which will correspond to a vertex $v \in T$.
The column sum will be the sum of the distances between $0$ and all distinct paths in the tree with $v$ as an endpoint.
By the symmetry of $k$-ary trees, we choose an arbitrary vertex $v$ at level $j\in [n]$, since the corresponding column sum will be the same as any other vertex on that level.

We claim that $j[n-j+1]_k+\sum_{i=1}^{j-1}(j-i)(k^{n-(j-i)})$ is the column sum corresponding to any vertex at level $j$.
For any vertex $c$ which is a child of $v$ (and including $v$), the distance between 0 and the path from $v$ to $c$ will be $j$.
There are
\[
1+k+k^{2}+ \cdots +k^{n-j} = \frac{k^{n-j+1}-1}{k-1}=[n-j+1]_k
\]
such vertices and each of these vertices contributes j to the column sum.
Hence in total these vertices contribute the first term to our claimed exponent.

Consider the path between $0$ and $v$, $P=\{v=w_0,w_1,w_2, \ldots, w_{j-1},0\}$.
For each $i\in [j-1]$, define $S_i$ to be the set of vertices which are in the subtree with $w_i$ as the root but not in the subtree with $w_{i-1}$ as the root.
The vertex $w_i$ is on level $j-i$ of $T$.
Thus, the subtree whose root is $w_i$ has $n-(j-i)$ levels of its own.
So, there are $k^{0}+k^1+ \cdots +k^{n-(j-i)}$ vertices in the subtree with $w_i$ as a root, and $k^0+k^1+ \cdots +k^{n-(j-i)-1}$ vertices in the subtree of $w_{i-1}$.
Thus, there are
\[
 k^{0}+k^1+ \cdots +k^{n-(j-i)}-(k^0+k^1+ \cdots +k^{n-(j-i)-1})=k^{n-(j-i)}
 \]
vertices in $S_i$.
By definition, the shortest distance between 0 and the path connecting $v$ and any vertex $s \in S_i$ is $j-i$, since the path connecting $v$ and $s$ passes through $w_i$, which must be the closest point on this path to $0$.

Therefore, for each $i \in [j-1]$, the sum of the columns corresponding to vertices in $S_i$ will be $(j-i)(k^{n-(j-i)})$.
So, for all vertices which are not children of $v$, the total column sums will be $\sum_{i=1}^{j-1}(j-i)(k^{n-(j-i)})$.
Hence,  $j[n-j+1]_k+\sum_{i=1}^{j-1}(j-i)(k^{n-(j-i)})$ is the column sum corresponding to any vertex at level $j\in [n]$.
Since there are $k^{j-1}$ vertices on each level, for a fixed $j$, the term
\[
1-q^{ \textstyle j[n-j+1]_k+k^{n-j}(j([j]_k-1)-k\frac{d}{dk}([j]_k-1))}
\]
will appear $k^{j-1}$ times as an exponent in the denominator of our generating function.
\end{proof}

When $k=2$, the formulas for our exponents simplify nicely, as follows.

\begin{corollary}
For a binary tree with $n$ levels,
\begin{eqnarray*}
\sigma_\C(q)&=&\frac{1}{\displaystyle \prod_{j=1}^{n} \left(1-q^{\textstyle j(2^{n-j+1}-1)+\sum_{i=1}^{j-1}(j-i)(2^{n-(j-i)})}\right)^{\textstyle 2^{j-1}}}\\
		   &=&\frac{1}{\displaystyle \prod_{j=1}^{n} \left(1-q^{\textstyle 2^{n-j+1}(2^{j}-1)-j}\right)^{\textstyle 2^{j-1}}}
\end{eqnarray*}
\end{corollary}

We omit the proof of the equivalence of the two rational forms in our Corollary, as they are the result of a tedious calculation.
For those interested in the calculation, we mention that the key step is the use of $\sum_{i=1}^{j-1}(j-i)2^i =\sum_{i=1}^{j-1}(2+2^2+\cdots+2^i)$.


\section{$n$-cycles}\label{ncycles}

In this section we consider the open problem due to Corteel, Savage, and Wilf given in Example~\ref{cycleconj}.
Recall that the Laplacian minor for a general $n$-cycle on vertices $1,2,\ldots,n$, labelled cyclically and minored at $n$, has the following form when rows and columns are indexed $1,2,\ldots,n-1$.
\[
L_{n,cyc}:=
\left( \begin{array}{cccccccc}
2 & -1 & 0 & 0 & \cdots & 0 & 0 & 0\\
-1 & 2 & -1  & 0 & \cdots & 0 & 0 & 0\\
0 & -1 & 2  & -1 & \cdots & 0 & 0 & 0\\
\vdots & \vdots & \vdots & \vdots  & \ddots & \vdots & \vdots & \vdots\\
0 & 0 & 0  & 0 & \cdots & -1 & 2 & -1\\
0 & 0 & 0 & 0 & \cdots & 0 & -1 & 2\\
   \end{array} \right) \, .
\] 
Define 
\[
\K_n:=\{\lambda\in \R^{n-1} : L_{n,cyc} \lambda \geq 0  \} \quad \text{and} \quad \sigma_{\K_n}(z_1,\ldots,z_{n-1}):=\sum_{\lambda\in \K_n\cap \Z^{n-1}}z_1^{\lambda_1}\cdots z_{n-1}^{\lambda_{n-1}} \, .
\]
Our goal is to study the integer point transform for $\K_n$ by enumerating the integer points in the fundamental parallelepiped for $\K_n$; along the way, we will see that there is a hidden additive structure to the ray generators of $\K_n$ which is useful.

By Lemma~\ref{facetraylemma}, the rows of $L_{n,cyc}$ determine the normal vectors to the facets of $\K_n$, while the columns of $L_{n,cyc}^{-1}$ provide (non-integral) ray generators of $\K_n$.
We first explicitly describe the entries in $L_{n,cyc}^{-1}$.
\begin{proposition}\label{ncycleinverse}
For an $n$-cycle and $1\leq i,j \leq n-1$, $L_{n,cyc}^{-1}=(b_{i,j})$ where
\[
  b_{i,j} = \left\{
  \begin{array}{l l}
    \frac{i(n-j)}{n} & \text{ if } i \leq j\\
    \frac{j(n-i)}{n} & \text{ if } i>j\\
  \end{array} \right. .
\]
\end{proposition}

\begin{proof}
Define the matrix $B=(b_{i,j})$ using the $b$-values given above.
Let $B\cdot L_{n,cyc}=(a_{i,j})$; since $L_{n,cyc}$ and $B$ are symmetric matrices, it follows that $a_{i,j}=a_{j,i}$.
Entry $a_{i,j}$ is given by $2b_{i,j}-b_{i,x}-b_{i,y}$ where $x$ and $y$ are the neighbors of $j$ in the cycle, where we drop the $b$-terms for $x$ or $y$ equal to $n$.
It is straightforward to show that $BL_{n,cyc}=I$ through a case-by-case analysis. 

As a representative example, for $i<j$ we have
\begin{align*}
a_{i,j} &=2b_{i,j}-b_{i,j-1}-b_{n,j+1} \\
&=2\left(\dfrac{i(n-j)}{n}\right)-\left(\dfrac{i(n-j+1)}{n}\right)-\left(\dfrac{i(n-j-1)}{n}\right) \\
&=0.
\end{align*}
\end{proof}

We next scale the matrix $L_{n,cyc}^{-1}$ by $n$ so that the columns are integral; in doing so, we find that there is group-theoretic structure to the ray generators of $\K_n$, in the following sense.

\begin{proposition}\label{cyclelaplacianmodn}
Let the matrix $M:=nL_{n,cyc}^{-1} \, \, (\bmod~n)$. Then
\[
M=
 \begin{pmatrix}
 | & | & \cdots & | & |\\
 v_{1} & v_{2} & \cdots & v_{n-2} & v_{n-1}\\
 |  & | & \cdots & | & |
 \end{pmatrix}
\]
where $v_1= (n-1,n-2,\ldots,2,1)^T \in \mathbb{Z}_n^{n-1}$ and $v_k=kv_1 \in \mathbb{Z}_n^{n-1}$.
\end{proposition}

\begin{proof}
Let $M=(m_{i,j})$.
Since $m_{i,j}$ is equal to either $i(n-j)$ or $j(n-i)$, it follows that $m_{i,j}\equiv -ij \bmod n$.
Consider column $v_1$.
We find
\begin{align*}
v_1&=\left(-1\cdot 1,-2\cdot 1,\dotsc,-(n-1)\cdot 1\right)^T \bmod n, \\
&=\left(n-1, n-2, \dotsc, 1\right)^T \in \mathbb{Z}_n^{n-1}.\\
\end{align*}
Evaluating an arbitrary column $v_k$, $1 \leq k \leq n-1$, we have
\begin{align*}
v_k&=\left(-1\cdot k,-2\cdot k,\dotsc,-(n-1)\cdot k\right)^T \bmod n, \\
&=k\cdot v_1 \in \mathbb{Z}_n^{n-1}.
\end{align*}
\end{proof}

Our next step is to use the columns of $nL_{n,cyc}^{-1}$ as ray generators for $\K_n$, leaving us only to understand the integer points in $\Pi_n$, the fundamental parallelepiped for $\K_n$.
For any integer point $\lambda\in \K_n$, since $L_{n,cyc}$ is an integer matrix, $L_{n,cyc}\lambda = c \in \Z^{n-1}$.
Thus, every integer point in $\K_n$ can be expressed in the form $L_{n,cyc}^{-1}c$ for an integral $c$; further, if $\lambda\in \Pi_n$, then
\[
\lambda = \left(nL_{n,cyc}^{-1}\right)\left(\frac{c}{n}\right)
\]
where the entries in $c$ are necessarily taken from $\{0,1,2,\ldots,n-1\}$.
The group-theoretic structure of $nL_{n,cyc}^{-1} \, \, (\bmod~n)$ allows us to classify these points.

\begin{proposition}
Suppose that the point $x=(x_1,x_2, \ldots, x_{n-1})\in \Pi_n$ is of the form $L_{n,cyc}^{-1}c$ for some vector $c$ with entries taken from $\{0,1,2,\ldots,n-1\}$.
Then $x=(x_1,x_2, \ldots, x_{n-1})\in \K_n$ is an integer point if and only if $x_1$ is an integer.
\end{proposition}

\begin{proof}
If $x$ is an integer vector, then $x_1$ is certainly an integer.  For the converse, suppose that $x_1$ is an integer. 
This happens precisely in the case where, for $l$ denoting the first row of $nL_n^{-1}$, we have that $l\cdot c = \sum_{j=1}^n l_{1,j}c_j$ is an integer multiple of $n$. 
Thus, $v_1 \cdot c \equiv \sum_{j=1}^n v_{1,j}c_j \equiv  0 \bmod{n}$ where $v_1$ is as defined in Proposition~\ref{cyclelaplacianmodn}.
Since $v_k=kv_1$,
\[
\sum_{j=1}^n v_{k,j}c_j \equiv \sum_{j=1}^n kv_{1,j}c_j \equiv k\sum_{j=1}^n v_{1,j}c_j \equiv  0 \bmod{n}.
\]
Therefore, $x_k$ is also an integer.
\end{proof}

Thus, the integer points in $\Pi_n$ are parametrized by the set of solutions to the system 
\[
\sum_{j=1}^{n-1}(n-j)c_j\equiv 0 \, (\bmod \, n)
\]
where $c_j\in \{0,1,2,\ldots,n-1\}$.
Equivalently, these points are parametrized by the set of partitions of multiples of $n$ into positive parts not exceeding $n-1$, with no more than $n-1$ of each part.
The solution set to this system of equations has arisen before in the study of ``Hermite reciprocity'' given in \cite{elashvilietal}.
Our next theorem follows directly from the preceding discussion.

\begin{theorem}\label{ncyclegenfnthm}
Let $S_n$ denote the set of solutions to the system $\sum_{j=1}^{n-1}(n-j)c_j\equiv 0 \, (\bmod \, n)$ where $c_j\in \{0,1,2,\ldots,n-1\}$.
Let $l_{i,j}$ denote the $i,j$-th entry of $nL_{n,cyc}^{-1}$.
Then
\[
\sigma_{\K_n}(z_1,\ldots,z_n) = \frac{\sum_{c\in S_n}\left(z_1^{\sum_{j=1}^nl_{1,j}c_j}z_2^{\sum_{j=1}^nl_{2,j}c_j}\cdots z_n^{\sum_{j=1}^nl_{n,j}c_j}\right)}{\prod_{j=1}^n(1-z_1^{l_{1,j}}z_2^{l_{2,j}}\cdots z_n^{l_{n,j}})} \, .
\]
\end{theorem}

While Theorem~\ref{ncyclegenfnthm} and its specializations resolve the question of Corteel, Savage, and Wilf, this resolution is not as clean as one might like.
In the next section, we show how a slight modification of the graph underlying this problem leads to much more simply stated and elegant, but still closely related, univariate generating function identities.


\section{Leafed $n$-cycles}\label{leafedncycles}

In this section, we examine the case of compositions constrained by $n$-cycles augmented with a leaf; we will refer to such a graph as a \emph{leafed $n$-cycle}.
Throughout this section, we label the leaf vertex as $n$ and label the vertices of the $n$-cycle cyclically as $0,1,\ldots,n-1$, where $n$ is adjacent to $0$ as depicted in Figure~\ref{ncycletailpic}.

\begin{figure}[ht]
 \includegraphics[width={5cm}]{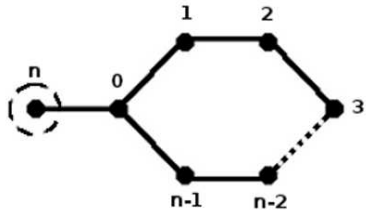}
 \caption{A leafed $n$-cycle}
 \label{ncycletailpic}
\end{figure}

For a leafed $n$-cycle, we denote by $L_n$ the minor of the Laplacian taken at vertex $n$ and observe that it has the form
\[
L_n \ = \ \left( \begin{array}{ccccccccc}
3 & -1 & 0 & 0 & 0 & \cdots & 0 & 0 & -1\\
-1 & 2 & -1 & 0 & 0 & \cdots & 0 & 0 & 0\\
0 & -1 & 2 & -1  & 0 & \cdots & 0 & 0 & 0\\
0 & 0 & -1 & 2  & -1 & \cdots & 0 & 0 & 0\\
\vdots & \vdots & \vdots & \vdots & \vdots  & \ddots & \vdots & \vdots & \vdots\\
0 & 0 & 0 & 0  & 0 & \cdots & -1 & 2 & -1\\
-1 & 0 & 0 & 0 & 0 & \cdots & 0 & -1 & 2\\
   \end{array} \right)
\]
where the columns (resp. rows) are labeled from top to bottom (resp. left to right) by $0,1,2,\ldots,n-1$.
Define 
\[
\C_n:=\{\lambda\in \R^n : L_n \lambda \geq 0  \} \quad \text{and} \quad \sigma_{\C_n}(q):=\sum_{\lambda\in \C_n\cap \Z^n}q^{\lambda_0} \, .
\]
While not immediately apparent, the reason for focusing on the first coordinate of integer points in $\C_n$ will become clear.


\subsection{A generating function identity}

Our main result regarding leafed cycles is the following generating function identity.
\begin{theorem}\label{leafedcyclethm}
Let $S_n$ denote the set of solutions to the system 
\[
0\cdot c_0 + \sum_{j=1}^{n-1}(n-j)c_j\equiv 0 \, (\bmod \, n) \, ,
\]
where $c_j\in \{0,1,2,\ldots,n-1\}$.
For $c\in S_n$, define $\phi(c):=\sum_{j=0}^{n-1}c_j$.
Then
\[
\sigma_{\C_n}(q)=\frac{\sum_{c\in S_n}q^{\phi(c)}}{(1-q^n)^n} \, .
\]
\end{theorem}
Note that $\sum_{c\in S_n}q^{\phi(c)}$ is the ordinary generating function for the ``number of parts'' statistic on the set of all non-negative partitions of multiples of $n$ into parts not exceeding $n-1$, with no more than $n-1$ of each part.
The similarity between Theorem~\ref{leafedcyclethm} and Theorem~\ref{ncyclegenfnthm} is immediate.
However, by switching from $n$-cycles to leafed $n$-cycles, a more elegant univariate identity is obtained than those obtained by straightforward specializations of Theorem~\ref{ncyclegenfnthm}.
We remark that our proof of Theorem~\ref{leafedcyclethm} can easily be extended to produce a multivariate statement analogous to Theorem~\ref{ncyclegenfnthm}; we omit this, since it is the specialization above that we find the most interesting consequence of the multivariate identity.

As in the previous section, our goal is to apply Lemma~\ref{fpplemma}.
We first explicitly describe the entries in $L_n^{-1}$, which are closely related to the entries of $L_{n,cyc}^{-1}$.

\begin{proposition}\label{ncycletail}
For a leafed $n$-cycle and $0\leq i,j \leq n-1$, $L_n^{-1}=(b_{i,j})$ where
\[
  b_{i,j} = \left\{
  \begin{array}{l l}
    \frac{i(n-j)}{n}+1 & \text{ if } i \leq j\\
    \frac{j(n-i)}{n}+1 & \text{ if } i>j\\
  \end{array} \right. .
\]
\end{proposition}

\begin{proof}
Define the matrix $B=(b_{i,j})$ using the $b$-values given above.
Let $L_nB=(a_{i,j})$; since $L_n$ and $B$ are symmetric matrices, it follows that $a_{i,j}=a_{j,i}$.
For $i\neq 0$, entry $a_{i,j}$ is given by $2b_{i,j}-b_{x,j}-b_{y,j}$ where $x$ and $y$ are the neighbors of $i$ in the cycle; the case $i=0$ is similar.
From these observations, it is straightforward to show that $L_nB=I$ through a case-by-case analysis. 

For example, for $j\neq 0$, $a_{0,j} = 0$ is shown as follows.
\begin{align*}
a_{0,j} &=3b_{0,j}-b_{1,j}-b_{n-1,j} \\
&=3\left(\dfrac{0(n-j)}{n}+1\right)-\left(\dfrac{1(n-j)}{n}+1\right)-\left(\dfrac{j(n-(n-1))}{n}+1\right) \\
&=0.
\end{align*}
\end{proof}

Note that every entry in the top row (and in the first column, by symmetry) of $L_n^{-1}$ is equal to $1$; it is this property of the inverse that allows us to effectively study the generating function recording only the first entry of each composition in $\C_n$.
We next scale the matrix $L_n^{-1}$ by $n$ so that the columns (i.e. ray generators of our cone) are integral; in doing so, we find that there is group-theoretic structure to the columns of $L_n^{-1}$, again related closely to that for $L_{n,cyc}^{-1}$.

\begin{proposition}\label{laplacianmodn}
Let the matrix $M=nL_n^{-1} (\bmod~n)$. Then
\[
M=
 \begin{pmatrix}
   | & | & | & \cdots & | & |\\
  v_0 & v_{1} & v_{2} & \cdots & v_{n-2} & v_{n-1}\\
|& |  & | & \cdots & | & |
 \end{pmatrix}
\]
where $v_1= (0,n-1,n-2,\ldots,2,1)^T \in \mathbb{Z}_n^n$ and $v_k=kv_1 \in \mathbb{Z}_n^n$.
\end{proposition}

\begin{proof}
Let $M=(m_{i,j})$.
Since $m_{i,j}$ is equal to either $i(n-j)+n$ or $j(n-i)+n$, it follows that $m_{i,j}\equiv -ij \bmod n$.
Consider column $v_1$.
We find
\begin{align*}
v_1&=\left(-0\cdot 1,-1\cdot 1,-2\cdot 1,\dotsc,-(n-1)\cdot 1\right)^T \bmod n, \\
&=\left(0, n-1, n-2, \dotsc, 1\right)^T \in \mathbb{Z}_n^n.\\
\end{align*}
Evaluating an arbitrary column $v_k$, $0 \leq k \leq n-1$, we have
\begin{align*}
v_k&=\left(-0\cdot k,-1\cdot k,-2\cdot k,\dotsc,-(n-1)\cdot k\right)^T \bmod n, \\
&=k\cdot v_1 \in \mathbb{Z}_n^n.
\end{align*}
\end{proof}

As in the case of $n$-cycles, our next step is to use the columns of $nL_{n}^{-1}$ as ray generators for $\C_n$; we must understand the integer points in $\Pi_n$, which now denotes the fundamental parallelepiped for $\C_n$.
For any integer point $\lambda\in \C_n$, since $L_{n}$ is an integer matrix, $L_{n}\lambda = c \in \Z^{n}$.
Thus, every integer point in $\C_n$ can be expressed in the form $L_{n}^{-1}c$ for an integral $c$.
As before, if $\lambda\in \Pi_n$, then
\[
\lambda = \left(nL_{n}^{-1}\right)\left(\frac{c}{n}\right)
\]
where the entries in $c$ are necessarily taken from $\{0,1,2,\ldots,n-1\}$.
The group-theoretic structure of $nL_{n}^{-1} \, \, (\bmod~n)$ allows us to classify these points.

\begin{proposition}
Suppose that the point $x=(x_0,x_1, \ldots, x_{n-1})\in \Pi_n$ is of the form $L_{n}^{-1}c$ for some vector $c$ with entries taken from $\{0,1,2,\ldots,n-1\}$.
Then $x=(x_0,x_1, \ldots, x_{n-1})\in \C_n$ is an integer point if and only if $x_1$ is an integer.
\end{proposition}

\begin{proof}
The proof is almost identical to that for the $n$-cycle case.
\end{proof}

Thus, the integer points in $\Pi_n$ are parametrized by the set of solutions to the system 
\[
0\cdot c_0 + \sum_{j=1}^{n-1}(n-j)c_j\equiv 0 \, (\bmod \, n) \, ,
\]
where $c_j\in \{0,1,2,\ldots,n-1\}$.
Equivalently, these points are parametrized by the set of partitions of multiples of $n$ into non-negative parts not exceeding $n-1$, with no more than $n-1$ of each part.

Our final observation in this subsection, and the key observation needed to apply Lemma~\ref{fpplemma} and complete the proof of Theorem~\ref{leafedcyclethm}, is that for any point $\lambda$ such that $L_n\lambda=c\in \Pi_n$, we have that $\lambda_0=\sum_ic_i$.
This is a consequence of Proposition~\ref{ncycletail}, which shows that the first row (and column) of $L_n^{-1}$ are vectors of all ones.


\subsection{Cyclically distinct compositions}

This subsection is devoted to the following conjecture, which provides a combinatorial interpretation for the coefficients of $\sigma_{\C_n}(q)$.
\begin{conjecture}\label{mainconjecture}
The number of integer points $\lambda\in\C_n$ with $\lambda_0=m$ is equal to the number of compositions of $m$ into $n$ parts that are cyclically distinct.
\end{conjecture}

Conjecture~\ref{mainconjecture} is based on experimental evidence for small values of $n$; further, we used LattE \cite{LatteM} to check the conjecture for all prime values of $n$ up to $n=17$.
An obvious way to attack Conjecture~\ref{mainconjecture} is to attempt to prove that for any $\lambda\in \C_n$ with $\lambda_0=m$, we have that $\lambda = L_n^{-1}c$ for an integer vector $c$ whose entries form a weak composition of $m$ that is cyclically distinct from all other such $c$ vectors forming elements of $\C_n$.
Unfortunately, this does not work out, as the following example demonstrates.

\begin{example}\label{cycdistex}
The weak compositions of $3$ into $3$ parts are given by the following families $F_1,F_2,F_3,$ and $F_4$; each family is the orbit of a single composition under the cyclic shift operation.
\[
\begin{array}{llll}
F_1: & (3,0,0) & (0,3,0) & (0,0,3)  \\
F_2: & (2,1,0) & (0,2,1) &  (1,0,2) \\
F_3: & (1,2,0) & (0,1,2) & (2,0,1) \\
F_4: & (1,1,1) & &  \\
\end{array}
\]
The four compositions which satisfy $L_3^{-1}c\in \C_n$ are the elements of $F_1$ and $F_4$.
Thus, while this shows that Conjecture~\ref{mainconjecture} holds in this case, it is not a consequence of the $c$-vectors being themselves cyclically distinct.
\end{example}

While this obvious approach to Conjecture~\ref{mainconjecture} fails in certain cases, it is successful in others, as seen in the following theorem.
\begin{theorem}\label{ncycle1tail} 
Let $n$ be prime and fix $m$ such that $\gcd(m,n)=1$.
The number of integer points in $\C_n$ with first coordinate $m$ is equal to the number of compositions of $m$ into $n$ parts that are cyclically distinct.
\end{theorem}

\begin{proof}
We must show that for every composition $c$ of $m$ with $n$ parts, exactly one of the cyclic shifts of $c$ will produce an integer point when $L_n^{-1}$ is applied. 
Begin with a composition $c=(c_0,c_1, \ldots, c_{n-1})$ of $m$ such that $L_n^{-1}c\in \C_n\cap \Z^n$.
Then for all $j$,
\[
\sum_{i=0}^{n-1} c_iv_{j,i}\equiv  0 \, (\bmod ~n)
\]
where $v_j=(v_{j,0},v_{j,1}, \ldots ,v_{j,n-1})$ is the $j$th row of $nL_n^{-1} \, (\bmod ~n)$.
Now consider another cyclic ordering $c'$ of $c$ where we shift the position of each entry by $k$, so $c_k'=(c_k, c_{k+1}, \ldots ,c_{k+n-1})$ with the indices taken $\bmod ~n$.
Then the $j$th coordinate of $L_n^{-1}c'$ is given by
\[
\sum_{i=0}^{n-1} c_{i+k}v_{j,i} = \sum_{i=0}^{n-1}c_iv_{j,i-k}.
\]
By properties of $v_j$ given in Proposition~\ref{laplacianmodn}, shifting the $i$th coordinate of $v_j$ by $k$ is equivalent to subtracting the row $v_k$ from the row $v_j$ component-wise. So,
\begin{align*}
\sum_{i=0}^{n-1}c_iv_{j,i-k} \equiv & \sum_{i=0}^{n-1}c_i(v_{j,i}-v_{j,k}) \\
\equiv & \sum_{i=0}^{n-1}c_iv_{j,i} - \sum_{i=0}^{n-1}c_iv_{j,k} \\
\equiv & \, \, 0 - v_{j,k} \sum_{i=0}^{n-1}c_i \\
\end{align*}
Since $\sum_{i=0}^{n-1}c_i \not\equiv   0 \bmod n$ by assumption, and since $v_{j,k} \sum_{i=0}^{n-1}c_i \not\equiv  0 \bmod n$ for $j,k \not\equiv  0$ due to $n$ being prime, we have that 
\[
\sum_{i=0}^{n-1}c_iv_{j,i-k} \not\equiv 0 \, .
\]
Thus, the non-trivial cyclic shifts of $c$ do not correspond to integer points in $\C_n$.
\end{proof}

It is straightforward to show that for the case where $n$ is prime and $n|m$, for any composition $c$ of $m$ into $n$ parts that satisfies $L_n^{-1}c\in \C_n\cap \Z^n$ all the cyclic shifts of $c$ also satisfy this condition.
Thus, the remaining challenge regarding Conjecture~\ref{mainconjecture} in the case where $n$ is prime is to show that this union of orbits of compositions under the cycle action contains as many compositions as there are orbits of the cycle action.
While we have not been successful in proving this, the study of these compositions reveals a beautiful underlying geometric structure for our cone, which we discuss in the following section.


\section{Reflexive polytopes and leafed cycles of prime length}\label{reflexive}

For leafed cycles of prime length, the remaining cases of Conjecture~\ref{mainconjecture} turn out to be equivalent to the problem of enumerating points in integer dilates of an integer polytope, which is the subject of Ehrhart theory.
The polytope in question is a simplex, and we will prove that it has the additional structure of being reflexive, a property we now introduce.

\subsection{Reflexive polytopes}

An $n$-dimensional \emph{polytope} $P$ in $\R^n$ may equivalently \cite[Chapter 1]{ZieglerLectures} be described as the convex hull of at least $n+1$ points in $\R^n$ or as the intersection of at least $n+1$ linear halfspaces in $\R^n$.
When $P=\mathrm{conv}(V)$ and $V$ is the minimal set such that this holds, then we say the points in $V$ are the \emph{vertices} of $P$.
If $P=\mathrm{conv}(V)$ and $|V|=n+1$, then we say $P$ is a \emph{simplex}.
We say $P$ is \emph{integral} if the vertices of $P$ are contained in $\Z^n$.

For $P\subset \R^n$ an integral polytope of dimension $n$, and for $t\in \mathbb{Z}_{>0}$, set $tP:=\left\{ tp : p\in P\right\}$ and $L_P(t):=|\mathbb{Z}^n\cap tP|$, i.e. the number of integer points in $tP$ is $L_P(t)$.
Define the \textit{Ehrhart series for $P$} to be $\mathrm{Ehr}_P(x):=1+ \sum_{t\geq 1}L_P(t)x^t$.
A fundamental theorem due to E. Ehrhart \cite{Ehrhart} states that for an $n$-dimensional integral polytope $P$ in $\mathbb{R}^n$, there exist complex values $h_j^*$ so that
\[
\mathrm{Ehr}_P(x)=\frac{\sum_{j=0}^{n}h_j^*x^j}{(1-x)^{n+1}} \phantom{...} \mathrm{and}\phantom{...} \sum_j h_j^*\neq 0 \, .
\]
A stronger result, originally due to R. Stanley in \cite{StanleyDecompositions}, is that the $h_j^*$'s are actually non-negative integers; Stanley's proof of this used commutative algebra, though several combinatorial and geometric proofs have since appeared.
We call the coefficient vector $h_P^*:=(h_0^*,\ldots,h_d^*)$ in the numerator of the rational generating function for $\mathrm{Ehr}_P(x)$ the \textit{$h$-star vector} for $P$.
The volume of $P$ can be recovered as $(\sum_j h_j^*)/n!$.

Obtaining a general understanding of the structure of $h^*$-vectors of integral polytopes is currently of great interest.
Recent activity, e.g. \cite{BatyrevDualPolyhedra,BeckHosten,BeyHenkWills,BraunEhrhartFormulaReflexivePolytopes,HibiDualPolytopes,HaaseMelnikov,MustataPayne,Payne}, has focused on the class of reflexive polytopes, which we now introduce as they will be needed in Section~\ref{reflexive}.
Given an $n$-dimensional polytope $P$, the \emph{polar} or \emph{dual} polytope to $P$ is
\[P^\Delta = \left\{x\in \mathbb{R}^{n}: x\cdot p \leq 1 \textrm{ for all } p\in P\right\} \, .
\]
Let $P^\circ$ denote the topological interior of $P$.

\begin{definition}
An $n$-dimensional polytope $P$ is \emph{reflexive} if $0\in P^\circ$ and both $P$ and $P^\Delta$ are integral.
\end{definition}

Reflexive polytopes have many rich properties, as seen in the following lemma.
\begin{lemma}{\rm \cite{BatyrevDualPolyhedra,HibiDualPolytopes}}\label{refpolylem} $P$ is reflexive if and only if $P$ is an integer polytope with $0\in P^\circ$ that satisfies one of the following (equivalent) conditions:
\begin{enumerate}
	\item $P^\Delta$ is integral.
	\item $L_{P^\circ}(t+1)=L_P(t)$ for all $t\in \mathbb{Z}_{\geq 0}$, i.e. all lattice points in $\mathbb{R}^{n}$ sit on the boundary of some non-negative integral dilate of $P$.
	\item $h_i^*=h_{n-i}^*$ for all $i$, where $h_i^*$ is the $i^{\textrm{th}}$ coefficient in the numerator of the Ehrhart series for $P$.
\end{enumerate}
\end{lemma}

Reflexive polytopes are simultaneously a very large class of integral polytopes and a very small one, in the following sense.
Due to a theorem of Lagarias and Ziegler \cite{lagariasziegler}, there are only finitely many reflexive polytopes (up to unimodular equivalence) in each dimension.
On the other hand, Haase and Melnikov \cite{HaaseMelnikov} proved that every integral polytope is a face of some reflexive polytope.
Because of this ``large yet small'' tension, combined with their importance in other areas of mathematics as noted before, whenever a reflexive polytope arises unexpectedly it is a surprise, an indication that something interesting is happening.


\subsection{Reflexive slices of Laplacian minor cones}

Our main result in this subsection is Theorem~\ref{reflexivethm}, which asserts that for prime $n$ values, reflexive polytopes arise as slices of the cone constrained by Laplacian minors for leafed $n$-cycles.
Note that when $n$ is not prime, this construction does not always yield reflexive simplices; this has been verified experimentally with LattE \cite{LatteM} and Lemma~\ref{refpolylem}.
However, exactly when the construction described below produces reflexive polytopes is at this time not clear.

Let $p$ be an odd prime and let $lCP_p$ be the simplex obtained by intersecting the cone constrained by a leafed $p$-cycle with the hyperplane $\lambda_1=p$ in $\R^p$.
\begin{theorem}\label{reflexivethm}
$lCP_p$ is reflexive (after translation by an integral vector).
\end{theorem}

\begin{proof}
By elementary results about polar polytopes found in \cite[Chapter 2]{ZieglerLectures}, $P$ is a reflexive polytope if and only if $P$ is integral, contains the origin in its interior, and has a half-space description of the form 
\[
\{x\in \R^d : Ax\geq -\mathbf{1}\}
\]
where $A$ is an integral matrix and $\mathbf{1}$ denotes the vector of all ones.

We first observe that $lCP_p$ is clearly an integer polytope, as its vertices are the columns of $pL_p^{-1}$, which are integral.
A halfspace description of $lCP_p$ is given by $\lambda_0 = p$ and 
\[
\left[
\begin{array}{ccccc}
3 & -1 & 0 & \cdots & -1 \\
-1 & 2 & -1 & \cdots & 0 \\
0 & -1 & 2 & \cdots & 0 \\
\vdots & \vdots & \vdots & \ddots & \vdots \\
-1 & 0 & 0 & \cdots & 2 \\
\end{array}
\right]
\left[
\begin{array}{c}
\lambda_0 \\
\lambda_1 \\
\lambda_2 \\
\vdots \\
\lambda_{p-1} \\
\end{array}
\right]
\geq
\left[
\begin{array}{c}
0 \\
0 \\
0 \\
\vdots \\
0 \\
\end{array}
\right] \, .
\]
Let $w_i$ be the $i$th column of $L_p^{-1}$, and observe that $L_p(\sum_iw_i)=L_p(\sum_iL_p^{-1}(e_i))=\mathbf{1}$.
Thus, $\sum_iw_i$ is an interior integer point of our simplex $lCP_p$; we want to change coordinates so that $\sum_iw_i$ is translated to the origin.
We therefore consider solutions to our halfspace description of the form $\lambda+\sum_iw_i$.
It follows that the description of our translated $lCP_p$ does not use a $\lambda_0$-coordinate (since $\lambda_0+p=p$ implies $\lambda_0=0$), and is given by
\[
\left[
\begin{array}{cccc}
-1 & 0 & \cdots & -1 \\
2 & -1 & \cdots & 0 \\
-1 & 2 & \cdots & 0 \\
\vdots & \vdots & \ddots & \vdots \\
0 & 0 & \cdots & 2 \\
\end{array}
\right]
\left[
\begin{array}{c}
\lambda_1 \\
\lambda_2 \\
\vdots \\
\lambda_{p-1} \\
\end{array}
\right]
\geq
\left[
\begin{array}{c}
-1 \\
-1 \\
\vdots \\
-1 \\
\end{array}
\right] \, .
\]
Since $lCP_p$ is an integral simplex, and $\sum_iw_i$ is an integer vector by which we translated $lCP_p$, our translated $lCP_p$ is still integral.
Thus $lCP_p$ is reflexive, and our proof is complete.
\end{proof}

While the proof above is elementary, it ties the study of compositions constrained by graph Laplacian minors into an interesting circle of questions regarding reflexive polytopes.
The only remaining case of Conjecture~\ref{mainconjecture} is that the number of integer points in $m\cdot lCP_p$ is equal to the number of compositions of $mp$ with $p$ parts, up to cyclic equivalence, which is now asking for a combinatorial interpretation for the Ehrhart series of a reflexive simplex.
Because of reflexivity, the generating function for this series yields a rational function with a symmetric numerator, which yields a functional relation on the Ehrhart polynomial for this simplex.

The study of reflexive polytopes goes hand in hand with the study of several other interesting classes of polytopes; one such example are normal polytopes, where an integer polytope $P$ is \emph{normal} if every integer point in the $m$-th dilate of $P$ is a sum of exactly $m$ integer points in $P$.
In our attempts to prove the prime case of Conjecture~\ref{mainconjecture}, we noticed that our techniques (though unsuccessful at providing a proof) provided evidence suggesting that $lCP_p$ is normal.
Normality is implied by the presence of a unimodular triangulation for an integral polytope; while we are not yet certain of the existence of such a triangulation for $lCP_p$ or of the normality of $lCP_p$, it would not surprise us if such a triangulation can be found for all prime $p$.

Our final remark regarding $lCP_p$, reflexivity, and normality, regards the unimodality of the $h^*$ vector of $lCP_p$.
It is a major open question (see \cite{Payne} and the references therein) whether all normal reflexive polytopes have unimodal $h^*$-vectors.
For $p\leq 7$, we found using LattE \cite{LatteM} that the $h^*$-vector for $lCP_p$ is unimodal.
We again suspect that this holds in general.


\section{Leafed cycles of length a power of two}\label{2powercycles}

We conclude our paper by considering leafed cycles for non-prime values of $n$; specifically, we study when $n=2^k$ for some $k$.
Using LattE \cite{LatteM}, one observes that the cone constrained by a leafed $8$-cycle does not have a reflexive $8$-th slice (we suspect this to be the case for all non-prime values of $n$).
Nevertheless, the integer point transform of this cone exhibits some interesting ``near-symmetry'' in the numerator for small powers of $2$.
Based on experimental evidence, we offer the following conjecture.
\begin{conjecture}
Let $\C_{2^k}$ be the cone constrained by the Laplacian minor of a leafed $n$-cycle where $n = 2^k$ for some integer $k \geq 2$. Then the generating function has the form
\[
\sigma_{\C_{2^k}}(q, 1, 1, \ldots, 1) \ = \ \dfrac{f(q)}{ \left(1-q^{\displaystyle{2^k}}\right) \  \displaystyle{\prod_{i=0}^{k-1}} \left(1 - q^{\displaystyle{2^{k-(i+1)}}}  \right)^{\displaystyle{2^i}}}
\]
where $f(q)$ has the following property. 
Let $(a_0, \ldots, a_j)$ denote the coefficient list of $f(q)$. 
If we append a $0$ to the end of this coefficient list, then take the difference between the appended coefficient list and its reverse, we obtain the coefficient list of the polynomial
\[
\sum_{i=0}^{n-2} {n-2 \choose i} (-1)^{\displaystyle{i}} q^{\displaystyle{ni}}.
\]
\end{conjecture}
The near-symmetry of the numerator polynomial indicates that there should be some interesting structure to the finite group obtained by quotienting the semigroup of integer points in $\C_{2^k}$ by the semigroup generated by a specific choice of integral ray generators for the cone.

\bibliographystyle{plain}
\bibliography{Braun}

\begin{thebibliography}{10}

\bibitem{andrewspa1}
George~E. Andrews.
\newblock Mac{M}ahon's partition analysis. {I}. {T}he lecture hall partition
  theorem.
\newblock In {\em Mathematical essays in honor of {G}ian-{C}arlo {R}ota
  ({C}ambridge, {MA}, 1996)}, volume 161 of {\em Progr. Math.}, pages 1--22.
  Birkh\"auser Boston, Boston, MA, 1998.

\bibitem{BatyrevDualPolyhedra}
Victor~V. Batyrev.
\newblock Dual polyhedra and mirror symmetry for {C}alabi-{Y}au hypersurfaces
  in toric varieties.
\newblock {\em J. Algebraic Geom.}, 3(3):493--535, 1994.

\bibitem{BeckBraunLe}
Matthias Beck, Benjamin Braun, and Nguyen Le.
\newblock Mahonian partition identities via polyhedral geometry.
\newblock {\em to appear in Developments in Mathematics, (memorial volume for
  Leon Ehrenpreis, edited by H. Farkas, R. Gunning, M. Knopp, and B. A.
  Taylor)}.
\newblock Preprint at http://arxiv.org/abs/1103.1070.

\bibitem{beckgesselleesavage}
Matthias Beck, Ira Gessel, Sunyoung Lee, and Carla Savage.
\newblock Symmetrically constrained compositions.
\newblock {\em Ramanujan J.}, 23:355--369, 2010.

\bibitem{BeckHosten}
Matthias Beck and Serkan Ho\c{s}ten.
\newblock Cyclotomic polytopes and growth series of cyclotomic lattices.
\newblock {\em Math. Res. Let.}, 13(4):607--622, 2006.

\bibitem{BeckRobinsCCD}
Matthias Beck and Sinai Robins.
\newblock {\em Computing the continuous discretely}.
\newblock Undergraduate Texts in Mathematics. Springer, New York, 2007.

\bibitem{BeyHenkWills}
Christian Bey, Martin Henk, and J{\"o}rg~M. Wills.
\newblock Notes on the roots of {E}hrhart polynomials.
\newblock {\em Discrete Comput. Geom.}, 38(1):81--98, 2007.

\bibitem{BiggsAGT}
Norman Biggs.
\newblock {\em Algebraic graph theory}.
\newblock Cambridge Mathematical Library. Cambridge University Press,
  Cambridge, second edition, 1993.

\bibitem{BraunEhrhartFormulaReflexivePolytopes}
Benjamin Braun.
\newblock An {E}hrhart series formula for reflexive polytopes.
\newblock {\em Electron. J. Combin.}, 13(1):Note 15, 5 pp. (electronic), 2006.

\bibitem{BeckBraunEulerMahonian}
Benjamin Braun and Matthias Beck.
\newblock Euler-{M}ahonian permutation statistics via polyhedral geometry.
\newblock Preprint.

\bibitem{BrightSavage}
Katie Bright and Carla Savage.
\newblock The geometry of lecture hall partitions and quadratic permutation
  statistics.
\newblock to appear in FPSAC proceedings, 2010.

\bibitem{corteelsavagewilf}
Sylvie Corteel, Carla~D. Savage, and Herbert~S. Wilf.
\newblock A note on partitions and compositions defined by inequalities.
\newblock {\em Integers}, 5(1):A24, 11 pp. (electronic), 2005.

\bibitem{SavageEtAlDigraphEnumeration}
J.~William Davis, Erwin D'Souza, Sunyoung Lee, and Carla~D. Savage.
\newblock Enumeration of integer solutions to linear inequalities defined by
  digraphs.
\newblock In {\em Integer points in polyhedra---geometry, number theory,
  representation theory, algebra, optimization, statistics}, volume 452 of {\em
  Contemp. Math.}, pages 79--91. Amer. Math. Soc., Providence, RI, 2008.

\bibitem{Ehrhart}
Eug{\`e}ne Ehrhart.
\newblock Sur les poly\`edres rationnels homoth\'etiques \`a {$n$}\ dimensions.
\newblock {\em C. R. Acad. Sci. Paris}, 254:616--618, 1962.

\bibitem{elashvilietal}
A.~Elashvili, M.~Jibladze, and D.~Pataraia.
\newblock Combinatorics of necklaces and ``{H}ermite reciprocity''.
\newblock {\em J. Algebraic Combin.}, 10(2):173--188, 1999.

\bibitem{HaaseMelnikov}
Christian Haase and Ilarion~V. Melnikov.
\newblock The reflexive dimension of a lattice polytope.
\newblock {\em Ann. Comb.}, 10(2):211--217, 2006.

\bibitem{HibiDualPolytopes}
Takayuki Hibi.
\newblock Dual polytopes of rational convex polytopes.
\newblock {\em Combinatorica}, 12(2):237--240, 1992.

\bibitem{LatteM}
Matthias K\"{o}ppe.
\newblock Latte macchiato, version 1.2-mk-0.6.
\newblock Available from URL http://www.math.ucdavis.edu/\%7emkoeppe/latte/,
  2009.

\bibitem{lagariasziegler}
Jeffrey~C. Lagarias and G{\"u}nter~M. Ziegler.
\newblock Bounds for lattice polytopes containing a fixed number of interior
  points in a sublattice.
\newblock {\em Canad. J. Math.}, 43(5):1022--1035, 1991.

\bibitem{MustataPayne}
Mircea Musta{\c{t}}{\v{a}} and Sam Payne.
\newblock Ehrhart polynomials and stringy {B}etti numbers.
\newblock {\em Math. Ann.}, 333(4):787--795, 2005.

\bibitem{PakGeomPart}
Igor Pak.
\newblock Partition identities and geometric bijections.
\newblock {\em Proc. Amer. Math. Soc.}, 132(12):3457--3462 (electronic), 2004.

\bibitem{Payne}
Sam Payne.
\newblock Ehrhart series and lattice triangulations.
\newblock {\em Discrete Comput. Geom.}, 40(3):365--376, 2008.

\bibitem{StanleyDecompositions}
Richard~P. Stanley.
\newblock Decompositions of rational convex polytopes.
\newblock {\em Ann. Discrete Math.}, 6:333--342, 1980.
\newblock Combinatorial mathematics, optimal designs and their applications
  (Proc. Sympos. Combin. Math. and Optimal Design, Colorado State Univ., Fort
  Collins, Colo., 1978).

\bibitem{StanleyVol1}
Richard~P. Stanley.
\newblock {\em Enumerative combinatorics. {V}ol. 1}, volume~49 of {\em
  Cambridge Studies in Advanced Mathematics}.
\newblock Cambridge University Press, Cambridge, 1997.
\newblock With a foreword by Gian-Carlo Rota, Corrected reprint of the 1986
  original.

\bibitem{ZieglerLectures}
G{\"u}nter~M. Ziegler.
\newblock {\em Lectures on polytopes}, volume 152 of {\em Graduate Texts in
  Mathematics}.
\newblock Springer-Verlag, New York, 1995.

\end{thebibliography}

\end{document}